\documentclass[a4paper,11pt,reqno]{amsart}

\usepackage{color}
\usepackage{enumerate}
\usepackage{amssymb, amsmath}

\theoremstyle{plain}
\newtheorem{theorem}{Theorem}[section]
\newtheorem{corollary}[theorem]{Corollary}
\newtheorem{lemma}[theorem]{Lemma}
\newtheorem{proposition}[theorem]{Proposition}
\newtheorem{definition}[theorem]{Definition}

\newtheorem*{definition*}{Definition}

\theoremstyle{remark}
\newtheorem{remark}[theorem]{Remark}

\newtheorem*{claim*}{Claim}
\newtheorem*{remark*}{Remark}
\newtheorem*{example*}{Example}
\newtheorem*{notation*}{Notation}

\def\N{{\mathbb N}}
\def\Z{{\mathbb Z}}
\def\R{{\mathbb R}}

\def\D{{\mathcal D}}

\newcommand{\Hess}{\mathrm{Hess}}
\newcommand{\Ent}{\mathrm{Ent}}

\newcommand{\E}{\mathcal{E}}

\newcommand{\Ric}{\mathrm{Ric}}
\newcommand{\RIC}{\vartheta}
\newcommand{\ric}{\eta}

\newcommand{\length}{\mathrm{Len}}
\newcommand{\action}{\mathrm{Act}}
\newcommand{\speed}{\mathrm{Speed}}

\newcommand{\supp}{\mathrm{supp}}
\newcommand{\Pz}{\mathcal{P}}

\begin{document}

\title
{Remarks about Synthetic Upper Ricci Bounds for Metric Measure Spaces}

\author{Karl-Theodor Sturm}

\maketitle

\begin{abstract}
We discuss various characterizations of synthetic upper Ricci bounds for metric measure spaces in terms of heat flow, entropy and optimal transport.
In particular, we present a characterization in terms of semiconcavity of the entropy along certain Wasserstein geodesics which  is stable under convergence of mm-spaces. 
And we prove that a related characterization is equivalent to an asymptotic lower bound on the growth of the Wasserstein distance between heat flows.
For weighted Riemannian manifolds, the crucial result will be a precise uniform two-sided bound for 
\begin{eqnarray*}\frac{d}{dt}\Big|_{t=0}W\big(\hat P_t\delta_x,\hat P_t\delta_y\big)\end{eqnarray*}
in terms of the mean value of the Bakry-Emery Ricci tensor $\Ric+\Hess f$ along the minimizing geodesic from $x$ to $y$ and an explicit correction term depending on the bound for the curvature along this curve.
\end{abstract}

%
%


\section{Introduction and Statement of Main Results}

We discuss various characterizations of  synthetic upper Ricci bounds for metric measure spaces. All of them are equivalent for (smooth) weighted Riemannian manifolds; in general they might diverge from one another but at least some implications hold.
\footnote{
The author gratefully acknowledges  support by the German Research Foundation through the Hausdorff Center for Mathematics and the Collaborative Research Center 1060 
as well as support by the European Union through the ERC-AdG ``RicciBounds''.
}

Let us briefly recall the Lott-Sturm-Villani definition of synthetic lower bounds for the Ricci curvature \cite{SturmActa1, LV1}. It will be used as a guideline for our synthetic characterizations of upper bounds for the Ricci curvature.
 For infinitesimally Hilbertian spaces $(X,d,m)$, the synthetic lower bound CD$(K,\infty)$  is defined as the  $K$-convexity of the Boltzmann entropy $S=\Ent(\,.\,|m)$ on the geodesic space $(\Pz(X),W)$:  
\begin{eqnarray}\label{ent-mid}
2 S(\rho^{1/2})\le S(\rho^0)+S(\rho^1)-\frac K4 W^2(\rho^0,\rho^1)\end{eqnarray}
for each geodesic  $\big(\rho^a\big)_{a\in[0,1]}$ in $\Pz(X)$.
Moreover \cite{AGS15AnnProb}, it is equivalent  to the contraction property
\begin{eqnarray}
W(\hat P_t\mu,\hat P_t\nu)\le e^{-Kt}W(\mu,\nu)\qquad (\forall \mu,\nu, \forall t>0).
\end{eqnarray}
Here and in the sequel, 
$W$ will denote the 2-Kantorovich-Wasserstein distance induced by $d$ and 
 $(\hat P_t)_{t>0}$ will denote the dual heat semigroup on $\Pz(X)$ associated with the mm-space $(X,d,m)$.
See the next section for more details.

The synthetic upper Ricci bounds to be presented below will reverse these inequalities: instead of $K$-convexity we will request $\kappa$-concavity, instead of $K$-contraction we will request $\kappa$-expansion.
However, in the case of upper bounds, this $\kappa$-expansion property will be requested only asymptotically for short times and only for Dirac measures as initial distributions. Similarly, the $\kappa$-concavity property will be requested only for
$W$-geodesics with endpoints supported in arbitrarily small neighborhoods of given points. And in both cases, the starting points should be close to each other.

A central role will be played by the quantity 
\begin{eqnarray}\RIC^+(x,y)=-\liminf_{t\to0}\frac1t\log\Big(W\big(\hat P_t\delta_x,\hat P_t\delta_y\big)/d(x,y)\Big)\end{eqnarray}
which measures the negative exponential rate of expansion for the Kantorovich-Wasserstein distance of two heat flows starting in $x$ and $y$, resp.
It is closely related to the quantity
\begin{eqnarray*}\Theta_p^+(\gamma)=-\frac1p\liminf_{t\to0}\frac1t\log\Big(\action_p(\mu_t^\cdot)/\action_p(\mu_0^\cdot)
\Big)\end{eqnarray*}
which measures the negative exponential rate of expansion for the $p$-action (w.r.t.\ the Kantorovich-Wasserstein distance $W$) of the curve $a\mapsto\mu_t^a:=\hat P_t\delta_{\gamma^a}$ in $\Pz(X)$ for any given curve $a\mapsto\gamma^a$ in $X$, evolving in time according to the dual heat flow. Recall that the $p$-action for $p=1$ is the length of the curve. For curves of constant speed, the quantity $\Theta_p^+(\gamma)$ is non-increasing in $p$. 
Analogously -- with $\liminf$ replaced by $\limsup$ -- we define the quantities $\RIC^-(x,y)$ and  $\Theta_p^-(\gamma)$.

\begin{definition} 
We say that a number $K$ is a strong synthetic upper Ricci bound for the mm-space $(X,d,m)$  if $\RIC^*(x)\le K$ for all $x\in X$ 
where $$\RIC^*(x):=\limsup_{y,z\to x}\RIC^+(y,z).$$

We say that a number $K$ is a weak synthetic upper Ricci bound for the mm-space $(X,d,m)$  if  $\Theta_1^+(\gamma)\le K$ for each geodesic $\gamma^\cdot$ in $X$.
\end{definition}

\begin{proposition}
For every geodesic $(\gamma^a)_{a\in[0,1]}$
$$\Theta_p^+(\gamma^\cdot)\le \int_0^1 \RIC^*(\gamma^a)da.$$
\end{proposition}
Thus in particular, every strong synthetic upper Ricci bound is a weak synthetic upper Ricci bound.

Our first main result will relate the quantity $\RIC^+$ to the integral of the Ricci curvature (or, more precisely, the Bakry-Emery Ricci curvature) $$\Ric_f=\Ric+\Hess f$$
in the case of weighted Riemannian manifolds $(M,g,f)$. 

\begin{theorem}\label{thm1} For all pairs of non-conjugate points $x,y\in M$
\begin{equation}\label{sharp}
\Ric_f(x,y)\le   \RIC^\pm(x,y)\le
\Ric_f(x,y)+\sigma(x,y)\cdot\tan^2\Big( \sqrt{\sigma(x,y)} d(x,y)/2\Big)\end{equation}
where
\begin{equation*}\Ric_f(x,y)=
 \int_{0}^1 \Ric_f(\dot\gamma^a, \dot\gamma^a)/ |\dot\gamma^a|^{2}\,da. \end{equation*} 
 denotes 
 the average Ricci curvature along  the (unique) geodesic  $\gamma=(\gamma^a)_{a\in[0,1]}$   from $x$ to $y$
and 
$\sigma(x,y)$ denotes the maximal modulus of the Riemannian curvature along this geodesic.
\end{theorem}

The assumption of $x$ and $y$ being non-conjugate is indispensable. For instance, if they were antipodal points on a cylinder then $\RIC^+(x,y)=\infty$.

As immediate consequences of the previous Theorem we obtain
$$\RIC^*(x)=\sup\big\{ \Ric_f(v,v)/|v|^2: \ v\in T_x(M)\big\}$$
for each $x\in X$. 
An improved argument yields 
$$-\partial_t\log\speed\big(\hat P_t \delta_{\exp_x(av)}\big)\Big|_{a=0,t=0}=\Ric_f(v,v)/|v|^2$$
for each $x\in $M and $v\in T_xM$ which then allows to extend a result of Gigli and Mantegazza \cite{GigliMa} in various respects: non-compact $M$, non-vanishing $f$, $p\not=2$ and $\gamma$ not necessarily being geodesic.

\begin{proposition}
For every $\mathcal C^1$-curve $\gamma$ in $M$ with constant speed 
and every $p\in [1,\infty)$
$$\Theta_p^\pm(\gamma)=
\int_{0}^1 \Ric_f(\dot\gamma^a, \dot\gamma^a)
\,da.$$ 
\end{proposition}

Next, we analyze our synthetic notions for Ricci bounds in the case of a cone. It will turn out that they will detect the unbounded positive curvature at the vertex.

\begin{proposition}
Let $X$ be the cone over the circle of length $\alpha<2\pi$. Then $\RIC^\pm(o,y)=\infty$  for all $y$ if $o$ denotes the vertex of the cone. Similarly, $\Theta_p^\pm(\gamma)=\infty$ for each geodesic emanating from the vertex.
\end{proposition}

A more general results will be proven in an upcoming paper with Erbar \cite{ErSt}:
{\it
Assume that $(X,d,m)$ satisfies the curvature-dimension condition RCD$(K',N')$ for some finite numbers $K',N'$ and that it is the $N$-cone over some mm-space. Then 
\begin{itemize}
\item[(i)] either  $\RIC^\pm(o,y)=\infty$  for all $y$ and $\Theta_p^\pm(\gamma)=\infty$ for each geodesic emanating from the vertex  $o$
\item[(ii)] or $N$ is an integer and $X=\R^{N+1}$. 
\end{itemize}
}

\bigskip

Our second main result -- which holds in the general setting of RCD-spaces -- relates the quantity $\RIC(x,y)$ (or more precisely, some relaxation of it) to the quantity $\ric(x,y)$ which measures the convexity/concavity of the Boltzmann entropy along transports from small neighborhoods of $x$ to small neighborhoods of $y$.
To be more precise, put 
\begin{eqnarray*}
\RIC^\flat(x,y)&:=&\lim_{\epsilon\to0} \ \inf\Big\{
-\partial_t^+  \log W\big( \hat P_t\mu,\hat P_t\nu\big)\big|_{t=0}: \
\supp[\mu]\subset B_\epsilon(x), \supp[\nu]\subset B_\epsilon(y)\Big\},\\
\ric(x,y)&:=&\lim_{\epsilon\to0}\  \inf\Big\{
\frac1{W^2(\rho^0,\rho^1)}\cdot\Big[ \partial^-_a S(\rho^a)\big|_{a=1}-\partial_a^+ S(\rho^a)\big|_{a=0}
\Big]: \ \big(\rho^a\big)_{a\in[0,1]} \mbox{ $W$-geodesic},
\nonumber\\
&&\qquad\qquad S(\rho^0)<\infty, \ S(\rho^1)<\infty, \  \supp[\rho^0]\subset B_\epsilon(x), \  \supp[\rho^1]\subset B_\epsilon(y)\Big\}.
\label{ric}
\end{eqnarray*}

\begin{theorem} For all $x,y\in X$ $$\RIC^\flat(x,y)=\ric(x,y).$$
\end{theorem}

\begin{remark} For weighted Riemannian manifolds
\begin{equation*}\label{sharp}
\Ric_f(x,y)\le   \RIC^\flat(x,y)
\le \Ric_f(x,y)+\sigma(x,y)\cdot\tan^2\Big( \sqrt{\sigma(x,y)} d(x,y)/2\Big)
\end{equation*}
provided $x$ and $y$ are not conjugate. Thus $\RIC^\flat$ may be used in an analogous way as $\RIC^\pm$ to characterize upper Ricci bounds. In particular, $\limsup_{y,z\to x} \RIC^\flat(y,z)\le K$ ($\forall x$) is equivalent to $\Ric_f\le K\cdot g$.
However, $\RIC^\flat$ will not be able to detect the positive curvature in the vertex of a cone.
\end{remark}

\begin{definition}
We say that  the mm-space $(X,d,m)$ has  robust synthetic upper Ricci bound $K$ if there exists an usc function $\overline K:(0,\infty)\to(-\infty,\infty]$ with $\lim_{r\to 0}\overline K(r)=K$  such that 
  for all $r>0$ with $\overline K(r)<\infty$,  all $x,y\in X$    with $d(x,y)=r$, and
     all $\mu^0\in\Pz(X)$ with $W(\mu^0,\delta_x)<r^4$  there exists a $W_2$-geodesic $(\mu^a)_{a\in[-1,1]}$ with $W(\mu^1,\delta_y)\le r^2$ and
\begin{equation}\label{robust}
\frac1{r^2}\Big(S(\mu^1)-2S(\mu^0)+S(\mu^{-1})\Big)\le \overline K(r).
\end{equation}
\end{definition}

A weighted Riemannian manifold with bounded geometry has  robust synthetic upper Ricci bound $K$ if and only if $\Ric_f\le K\cdot g$.
\begin{theorem} If a sequence of mm-spaces $(X_n,d_n,m_n), n\in\N$, converges to a locally compact mm-space  $(X,d,m)$ w.r.t.\ mGH (or, more generally, w.r.t.\ the $L^2$-transportation distance $\mathbb D_2$) and if each element of the sequence has  robust synthetic upper Ricci bound $K$ with the same approximation function $\overline K$ then also the limit space has  robust synthetic upper Ricci bound $K$ with  approximation function $\overline K$.

\end{theorem}

\bigskip

\noindent
{\bf Some notations.}
We use $\partial_t$ as a short hand notation for $\frac{d}{dt}$.
Moreover, we put
$\partial_t^+u(t)=\limsup_{s\to t}\frac1{t-s}(u(t)-u(s))$ and
$\partial_t^-u(t)=\liminf_{s\to t}\frac1{t-s}(u(t)-u(s))$.

\section{Synthetic Ricci Bounds for Metric Measure Spaces}
Throughout this paper, a \emph{metric measure space} (or briefly mm-space) will be a triple $(X,d,m)$ 
consisting of a Polish space $X$ equipped with a complete separable metric $d$ (inducing the topology of $X$) and a measure $m$ on the Borel $\sigma$-algebra of $X$. We 
will always assume that $m$ has full topological support and that $\int e^{-C d^2(x,z)}\,dm(x)<\infty$ for some $C\in\R$ and $z\in X$.

$W$ or $W_2$ 
will denote the $L^2$-Kantorovich-Wasserstein distance on the space $\Pz(X)$ of all probability measures on $X$. The Boltzmann entropy w.r.t.\ $m$ will be denoted by $S:\Pz(X)\to [-\infty,\infty]$. That is,
$$S(\mu)=\int u\log u\,dm$$
if $\mu$ is absolutely continuous with density $u$ and 
$\int (u\log u)_+\,dm<\infty$; otherwise
$S(\mu)=\infty$.

For $u\in L^2(X,m)$, the \emph{energy} (or Cheeger energy) can be represented 
as
$\E(u)=\frac12\int |Du|^2_*\,dm$
where $|Du|_*$ denotes the  minimal weak upper gradient of $u$.
The gradient flow for $\E$ on $L^2(X,m)$ defines uniquely the \emph{heat semigroup}
$$P_t: L^2(X,m)\to L^2(X,m).$$ 
Except in the last chapter, we always assume that this semigroup of operators extends 
to a 
 semigroup of Markov kernels
$\hat P_t: \Pz(X)\to \Pz(X)$  (`dual heat semigroup') such that
$\hat P_t(u\,m)=(P_t u)\, m$
for each nonnegative $u\in L^2(X,m)$ with $\int u\,dm=1$ and
 $\hat P_t\mu(.)=\int \hat P_t\delta_x(.)\, \mu(dx)$ for each $\mu\in\Pz(X)$.
In particular, this will be the case if $(X,d,m)$ satisfies an RCD$(K,\infty)$-condition for some $K\in\R$,
  \cite{AGS14Duke}.
 
We define functions $\RIC^+$, $\RIC^-$, and  $\RIC^\flat$  on $X\times X$ by
\begin{eqnarray*}\RIC^\pm(x,y)&:=&-\partial_t^\mp\log\Big(W\big(\hat P_t\delta_x,\hat P_t\delta_y\big)\Big)
\Big|_{t=0}\\
\RIC^\flat(x,y)&:=&\lim_{\epsilon\to0} \inf\Big\{
-\partial_t^+  \log W\big( \hat P_t\mu,\hat P_t\nu\big)\big|_{t=0}: \
\supp[\mu]\subset B_\epsilon(x), \supp[\nu]\subset B_\epsilon(y)\Big\}.
\end{eqnarray*}
Obviously, 
$$\RIC^\flat(x,y)\le \liminf_{(x',y')\to (x,y)}\RIC^-(x',y')\le \RIC^-(x,y)\le\RIC^+(x,y).$$

\begin{definition}
We say that a number $K$ is a strong synthetic upper Ricci bound for the mm-space $(X,d,m)$  if $\RIC^*(x)\le K$ for all $x\in X$ 
where
$$\RIC^*(x):=\limsup_{y,z\to x}\RIC^+(y,z).$$
\end{definition}

\subsection{Evolution of Curves under the Dual Heat Flow}

 For $p\in [1,\infty)$ we define the $p$-action 
of a curve $\gamma: [0,1]\to X, a\mapsto \gamma^a$
by
$$\action_p(\gamma^\cdot)
=\sup_{(a_i)_i} \sum_i\frac1{|a_i-a_{i-1}|^{p-1}}d(\gamma^{a_{i-1}},\gamma^{a_i})^p
=\limsup_{(a_i)_i} \sum_i\frac1{|a_i-a_{i-1}|^{p-1}}d(\gamma^{a_{i-1}},\gamma^{a_i})^p$$
where the  $\sup$ is taken over all partitions   $0=a_0<a_1<\ldots<a_n=1$ and the $\limsup$ over those with  mesh size going to 0. For $p=1$ this is the length $\length(\gamma^\cdot):=\action_1(\gamma^\cdot)$.
Analogously (with $W$ in the place of $d$), we define length and $p$-action for curves in $\Pz(X)$.

We will be particularly interested in the growth of length (and action) of curves which evolve under the dual heat flow.
Here any curve $a\mapsto\gamma^a$ in $X$ will be identified with the curve
$a\mapsto \mu_0^a:=\delta_{\gamma^a}$ in $\Pz(X)$ and its evolution under the dual heat flow is the curve $a\mapsto\mu_t^a:=\hat P_t\mu_0^a$.

\begin{definition}
We say that a number $K$ is a weak synthetic upper Ricci bound for the mm-space $(X,d,m)$  if  $\Theta^+(\gamma)\le K$ for each geodesic $\gamma^\cdot$ in $X$ where
$$\Theta^+(\gamma):=-\partial_t^-\log\big(\length(\hat P_t\delta_{\gamma^\cdot}
\big)\big|_{t=0}.$$
\end{definition}

\begin{remark} Slightly extending the previous definition, we say that a number $K$ is a $p$-weak synthetic upper Ricci bound for the mm-space $(X,d,m)$  if  $\Theta_p^+(\gamma)\le K$ for each geodesic $\gamma^\cdot$ in $X$ where
$$\Theta_p^+(\gamma):=-\frac1p\partial_t^-\log\big(\action_p(\hat P_t\delta_{\gamma^\cdot}
\big)\big|_{t=0}.$$
Since $\action_p(\mu^\cdot)\ge \big(\length(\mu^\cdot)\big)^p$ for each curve, with equality for each constant speed curve  and thus
$$-\frac1p\partial_t^-\log\big(\action_p(\hat P_t\delta_{\gamma^\cdot}\big)\big|_{t=0}\le -\partial_t^-\log\big(\length(\hat P_t\delta_{\gamma^\cdot}
\big)\big|_{t=0},$$
every weak synthetic upper Ricci bound is also a $p$-weak synthetic upper Ricci bound.
\end{remark}

According to the following Theorem, any strong synthetic upper Ricci bound is a weak synthetic upper Ricci bound.

\begin{theorem}
For every geodesic $(\gamma^a)_{a\in[0,1]}$ in $X$ and  every $p\in[1,\infty)$, the decay of the  $p$-action of the curve $a\mapsto \mu_t^a:=\hat P_t\delta_{\gamma^a}$ in  $(\Pz(X),W)$ is controlled by
\begin{equation}
-\frac1p\partial_t^- \log\action_p(\mu_t^\cdot)\big|_{t=0}\le \int_0^1 \RIC^*(\gamma^a)da.
\end{equation}
In particular,
$$\RIC^+(\gamma)\le\int_0^1 \RIC^*(\gamma^a)da.$$
\end{theorem}

\begin{proof} Given $\gamma$ as above and $\delta>0$, by upper semicontinuity of $\RIC^*$ there exists a finite partition $0=a_0<a_1<\ldots<a_n=1$ such that   $\forall i$
$$\RIC^+(\gamma^{a_i},\gamma^{a_{i-1}})\le \frac1{|a_i-a_{i-1}|}\int_{a_{i-1}}^{a_i}\RIC^*(\gamma^a)da+\delta.$$
Fixing this finite partition,  there exists $t_0>0$ s.t. $\forall t\in (0,t_0), \forall i$
$$W\big(\mu_t^{a_i}, \mu_t^{a_{i-1}}\big)\ge e^{-\big(\RIC^+(\gamma^{a_i},\gamma^{a_{i-1}})+\delta\big)t}\cdot d\big(\gamma^{a_i}, \gamma^{a_{i-1}}\big).$$
Therefore,
\begin{eqnarray*}
\action_p(\mu_t^\cdot)&\ge&
\sum_i \frac1{|a_i-a_{i-1}|^{p-1}}W\big(\mu_t^{a_i}, \mu_t^{a_{i-1}}\big)^p\\
&\ge&
\sum_i \frac1{|a_i-a_{i-1}|^{p-1}}\Big[1-\Big(\RIC^+(\gamma^{a_i},\gamma^{a_{i-1}})+\delta\Big)pt\Big]\cdot d\big(\gamma^{a_i}, \gamma^{a_{i-1}}\big)^p\\
&=&
\action_p(\gamma^\cdot)\cdot \Big[1-\delta pt -pt\cdot \sum_i
\RIC^+(\gamma^{a_i},\gamma^{a_{i-1}})\cdot |a_i-a_{i-1}|\Big]\\
&\ge&
\action_p(\gamma^\cdot)\cdot \Big[1-2\delta pt -pt\cdot \int_0^1
\RIC^*(\gamma^a)\,da\Big]
\end{eqnarray*}
and thus
$$-\partial_t^- \log\action_p(\mu_t^\cdot)\big|_{t=0}\le2\delta p + p\, \int_0^1
\RIC^*(\gamma^a)\,da.$$
Since $\delta>0$ was arbitrary this proves the claim.
\end{proof}

\subsection{Example: The Cone} 
 
 Let us consider in detail  the mm-space $(X,d,m)$   where $X$ is the cone over a circle of  length $\alpha<2\pi$, equipped with the cone metric $d$ and the two-dimensional Lebesgue measure $m$.
 Denote the vertex by $o$. Then the `punctured cone' $X\setminus \{o\}$ is a Riemannian manifold with vanishing Ricci curvature.

\begin{theorem}  
\begin{itemize}
\item[(i)] $\RIC^\pm(x,z)=\infty$ for all $x\in X$
\item[(ii)] $\RIC^\pm(x,y)=0$ for all $x,y\in X\setminus \{o\}$
\item[(iii)] $\partial_t \length(\hat P_t\delta_{\gamma^\cdot})\big|_{t=0}=-\infty$
for each geodesic $(\gamma^a)_{a\in[0,1]}$ in $X$ which emanates or terminates in the vertex (i.e.
 $\gamma^0=o$ or $\gamma^1=o$)
 \item[(iv)]   $\partial_t \length(\hat P_t\delta_{\gamma^\cdot})\big|_{t=0}=0$ for all other geodesic $(\gamma^a)_{a\in[0,1]}$ in $X$.
\end{itemize}
\end{theorem}

\begin{proof}
The assertions  immediately follow from the following precise heat kernel asymptotic on the cone. 
\end{proof}

\begin{corollary} The cone does not admit any strong synthetic upper Ricci bound nor any weak synthetic upper Ricci bound.
\end{corollary}

\begin{remark} In contrast to 
$\RIC^+$ and $\RIC^-$, the quantity $\RIC^\flat$ 
will not be able to detect the positive curvature of the cone which is concentrated in one point, the vertex.
Indeed, 
  the equivalence $\RIC^\flat=\ric$ (Theorem \ref{thm-prec}) and the fact that no optimal transport will move mass through the vertex will imply
  $\RIC^\flat(x,y)=0$ for all $x,y\in X$.
\end{remark}

\begin{lemma}  \label{cone}
For the cone over the circle of length $\alpha<2\pi$
\begin{equation*}
W\big(\hat P_t\delta_x,  \hat P_t\delta_y\big)=\left\{
\begin{array}{ll}
d(x,y){-\sqrt{\pi t}}\cdot \frac2\alpha\sin\frac\alpha2 + {\it O}( t), \quad \mbox{if $x$ or $y$ is the vertex}\\
d(x,y)+{\it o}( t), \quad \mbox{else}.
\end{array}\right.
\end{equation*}
\end{lemma}

\begin{proof}
The claim for $x,y$ away from the vertex follows from Riemannian calculations and the fact that the punctured cone is flat.

Now let us assume that $x$ is in the vertex.
We introduce coordinates on $X$ (ignoring some double labelling) such that
$$X=\{(z_1,z_2)\in\R^2: \cot\frac\alpha2 \cdot |z_2|\le z_1\}=\{(r\cos\phi,r\sin\phi): r\in\R_+, \phi\in [-\frac\alpha2,\frac\alpha2]\}$$
and $x=(0,0), y=(R,0)$. 
Then  $$\hat P_t\delta_x(dz)=\frac1{2\alpha t}e^{-\frac{r^2}{4t}}r\,dr\,d\phi$$ 
for $z=(r\cos\phi,r\sin\phi)$ (`exact Gaussian')
whereas $\hat P_t\delta_y$ for $t\to0$ can be approximated by a Gaussian
$$\hat P_t\delta_y(dz)=\frac1{4\pi t}e^{-\frac{r^2}{4t}}r\,dr\,d\psi+{\it O}(t)$$ 
for $z=(R+r\cos\psi,r\sin\psi)$.
For the optimal coupling of them (for fixed $R>0$ and $t\approx 0$) the $z_2$-distributions will be irrelevant (more precisely, their contribution will be ${\it O}(t)$).
The optimal coupling of the $z_1$-distributions is given by pairing $((R+r)\cos\psi,(R+r)\sin\psi)$ with
$(r\cos(\frac\alpha{2\pi}\psi),r\sin(\frac\alpha{2\pi}\psi))$ for $r>0$ and $\psi\in [-\pi,\pi]$. Therefore
\begin{eqnarray*}
W^2\big( \hat P_t\delta_x,  \hat P_t\delta_y\big)&=&\frac1{4\pi t}\int_0^\infty
\int_{-\pi}^\pi \big| R+r\cos\psi-r\cos (\frac\alpha{2\pi}\psi)\big|^2d\psi
\, e^{-\frac{r^2}{4t}}r\,dr+{\it O}(t)\\
&=&
\frac1{4\pi t}\int_0^\infty
\int_{-\pi}^\pi \big[ R^2-2Rr\cos (\frac\alpha{2\pi}\psi)\big]d\psi
\, e^{-\frac{r^2}{4t}}r\,dr+{\it O}(t)\\
&=& R^2-2R\sqrt{\pi t} \cdot \frac2\alpha\sin\frac\alpha2+{\it O}(t)
\end{eqnarray*} and thus
\begin{eqnarray*}
W\big( \hat P_t\delta_x,  \hat P_t\delta_y\big)
&=& R-\sqrt{\pi t} \cdot \frac2\alpha\sin\frac\alpha2+{\it O}(t).
\end{eqnarray*} 
\end{proof}

In a follow-up paper together with Erbar \cite{ErSt} we could deduce a far reaching generalization of the previous Theorem.

\begin{theorem} 
Assume that $(X,d,m)$ satisfies the curvature-dimension condition RCD$(K',N')$ for some finite numbers $K',N'$ and that it is the $N$-cone over some mm-space $(X',d',m')$. Then 
\begin{itemize}
\item[(i)] either  $\RIC^\pm(o,y)=\infty$  for all $y$ and $\RIC^\pm(\gamma)=\infty$ for each geodesic emanating from the tip 
\item[(ii)] or $N$ is an integer and $(X,d,m)$ is isomorphic to $(\R^{N+1}, |\,.\,|, dx)$. 
\end{itemize}
In particular, (up to isomorphism) the only Ricci bounded $N$-cone among all mm-spaces is $(\R^{N+1}, |\,.\,|, dx)$. 
\end{theorem}

\subsection{Synthetic Lower Ricci Bounds}

\begin{theorem} For any infinitesimally Hilbertian mm-space $(X,d,m)$ and any $p\in [1,\infty)$ the following are equivalent
\begin{enumerate}
\item $(X,d,m)$ satisfies RCD$(K,\infty)$
\item  for all $x,y\in X$ and all $t>0$
$$W(\hat P_t\delta_x, \hat P_t\delta_y)\le e^{-Kt} d(x,y)$$
\item  for all $\mu,\nu\in\Pz(X)$ and all $t>0$
$$W(\hat P_t\mu, \hat P_t\nu)\le e^{-Kt} W(\mu,\nu)$$
\item  for all $\mu,\nu\in\Pz(X)$,  all $t>0$ and all $q\in[1,\infty]$
$$W_q(\hat P_t\mu, \hat P_t\nu)\le e^{-Kt} W_q(\mu,\nu)$$
\item  for all curves $(\mu^a)_{a\in [0,1]}$ in $\Pz(X)$, all $t>0$ and all $q\in[1,\infty]$
$$\action_{p,q}(\hat P_t\mu^\cdot)\le e^{-Kpt} \action_{p,q}(\mu^\cdot)$$
\item $t\mapsto\length(\hat P_t\mu^\cdot)$ is absolutely upper continuous 
for all geodesics $(\mu^a)_{a\in [0,1]}$ in $\Pz(X)$ with
$$-\partial^+_t|_{t=0}\log\length(\hat P_t\mu^\cdot)\ge K$$ 
\item $t\mapsto W(\hat P_t\mu, \hat P_t\nu)$ is absolutely upper continuous 
for all $\mu,\nu\in\Pz(X)$
with $$-\partial^+_t|_{t=0}\log W(\hat P_t\mu, \hat P_t\nu)\ge K.$$
\end{enumerate}
\end{theorem}

Here $W_q$ denotes the $q$-Kantorovich-Wasserstein distance and $\action_{p,q}$ denotes the $p$-action w.r.t.\ $W_q$, that is,
$$\action_{p,q}(\mu^\cdot)=\sup_{(a_i)}\sum_i \frac1{|a_i-a_{i-1}|^{p-1}}W_q\big(\mu^{a^i}, \mu^{a^{i-1}}\big)^p$$
where the supremum is taken over all finite partitions $0=a_0<a_1<\ldots<a_n=1$.
\begin{proof}
Various of the implications are obvious: (4) $\Rightarrow$(3)$\Rightarrow$ (2) and
(5) $\Rightarrow$ (6)$\Rightarrow$(7).

(1) $\Leftrightarrow$  (2): This is the basic equivalence as deduced by Ambrosio, Gigli, Savar\'e \cite{AGS15AnnProb}.

(2) $\Rightarrow$ (3): Given $\mu,\nu$, choose a $W$-optimal coupling  $q$ of them. Then the measure $$q_t(dx',dy'):=\int\hat P_t\delta_x(dx')\hat P_t\delta_y(dy') q(dx,dy)$$ is a coupling of $\hat P_t\mu(dx')=\int \hat P_t\delta_x(dx')\mu(dx)$ and $\hat P_t\nu(dy')=\int \hat P_t\delta_y(dy')\nu(dy)$. Thus 
\begin{eqnarray*}W^2(\hat P_t\mu, \hat P_t\nu)&\le&\int d^2(x',y') \int\hat P_t\delta_x(dx')\hat P_t\delta_y(dy') q(dx,dy)\\
&\le& e^{-2Kt}\int d^2(x,y)q(dx,dy)=e^{-2Kt} W^2(\mu,\nu)
\end{eqnarray*}
where we made use of (2) for the last estimate.

(3) $\Rightarrow$ (4): This is the self-improvement property as deduced by Savar\'e \cite{Sav}.

(4) $\Rightarrow$ (5): Given the curve $(\mu^a)_{a\in [0,1]}$, numbers $t,p,q$ as above as well as any finite partition $0\le a_0<a_1,\ldots , a_n=1$ of the parameter interval $[0,1]$
\begin{eqnarray*}\sum_i \frac1{|a_i-a_{i-1}|^{p-1}}W_q\big(\hat P_t\mu^{a^i},\hat P_t\mu^{a^{i-1}}\big)^p&\le& e^{-Kpt}
\sum_i \frac1{|a_i-a_{i-1}|^{p-1}}W_q\big(\mu^{a^i},\hat \mu^{a^{i-1}}\big)^p\\
&\le& e^{-Kpt} \action_{p,q}(\mu^\cdot).\end{eqnarray*}
Passing to the supremum over all partitions, yields the claim.

(7) $\Rightarrow$ (2):  Integrating (7) from 0 to $t$, we to obtain (2):
$$e^{Kt}W(\hat P_t\mu, \hat P_t\nu)-W(\mu, \nu)\le\int_0^t
\partial_s\Big(e^{Ks}W(\hat P_s\mu, \hat P_s\nu)\Big)ds\le0.$$
\end{proof}

\begin{proposition} Assume that a priori $(X,d,m)$ is known to satisfy some RCD$(K',\infty)$-condition. Then in addition the following assertions are equivalent to the previous assertions (1) - (7):
\begin{enumerate}
\item[(8)] 
 $-\partial^+_t|_{t=0}\log\length(\hat P\delta_{\gamma^\cdot})\ge K$ for all geodesics $(\gamma^a)_{a\in [0,1]}$ in $X$
 \item[(9)] 
  $-\partial^+_t|_{t=0}\log W(\hat P_t\delta_x, \hat P_t\delta_y)\ge K$ for all $x,y\in X$.
\end{enumerate}
\end{proposition}

\begin{proof}
The implications 
(6) $\Rightarrow$(8)$\Rightarrow$ (9) are obvious. 

(9) $\Rightarrow$(7): Given $\mu,\nu$, choose a $W$-optimal coupling  $q$ of them. Without restriction, we may assume $\int d^2(x,y)q(dx,dy)<\infty$.
Then by Fatou's lemma (since the following integrand  is bounded from above by $e^{2(K-K')t}d^2(x,y)$ which is integrable w.r.t.\ $q$)
\begin{eqnarray*}
\partial^+_t|_{t=0}\big[e^{2Kt}W^2(\hat P_t\mu, \hat P_t\nu)\big]&\le&
\limsup_{t\to0}\int \frac1t \Big[\big[e^{2Kt}W^2(\hat P_t\delta_x, \hat P_t\delta_y)\big]-d^2(x,y)\Big]dq(dx,dy)\\
&\le&\int \limsup_{t\to0}\frac1t \Big[\big[e^{2Kt}W^2(\hat P_t\delta_x, \hat P_t\delta_y)\big]-d^2(x,y)\Big]dq(dx,dy)
\le0.
\end{eqnarray*}
The absolute continuity of $t\mapsto W(\hat P_t\mu, \hat P_t\nu)$ follows from the RCD$(K',\infty)$-assumption.
\end{proof}

\section{Ricci Bounds for Weighted Riemannian Manifolds}

In this section, we will discuss the previously introduced synthetic concepts for Ricci bounds in the more restrictive setting of weighted Riemannian manifolds.
A \emph{weighted Riemannian manifold} is a triple $(M,g,f)$ where $M$ is a (finite dimensional, smooth) manifold equipped with a Riemannian metric tensor $g$ as well as with a smooth `weight' function $f:M\to\R$. It canonically defines a metric measure space $(X,d,m)$ with $X=M$ where $d=d_g$ is the Riemannian distance induced by $g$ and $m=e^{-f}\mbox{vol}_g$ is the Borel measure on $M$ with density $e^{-f}$ w.r.t.\ the Riemannian volume measure $\mbox{vol}_g$ induced by $g$. The induced mm-space is always infinitesimally Hilbertian.
The mm-space induced by $(M,g,f)$ satisfies a curvature-dimension condition RCD$(K,\infty)$ if and only if $\Ric_f\ge K\cdot g$ (in the sense of inequalities between quadratic forms on the tangent bundle of $M$) where 
 $$\Ric_f=\Ric+\Hess f$$ denotes the `canonical'  Ricci tensor on the weighted Riemannian manifold.
The `canonical' Laplacian will be  $\Delta_f=\Delta-\nabla f\cdot\nabla$.
For each $\mathcal C^1$-curve $(\gamma^a)_{a\in [b,c]}$ in $M$ we denote the 
the mean value of  
the normalized Ricci curvature along the curve
by
\begin{equation}\label{Kxy}\Ric_f(\gamma):=
\frac1{c-b} \int_{b}^c \Ric_f(\dot\gamma^a, \dot\gamma^a)/ |\dot\gamma^a|^{2}\,da. \end{equation} 

Throughout the sequel, we will assume that the mm-space induced by $(M,g,f)$ satisfies a curvature-dimension condition RCD$(K',N)$ for some pair of finite numbers $K',N$.

\subsection{Precise Estimates for Heat Flows}

\begin{theorem}\label{thm-sharp}
For all pairs of non-conjugate points $x,y\in M$
\begin{equation}\label{sharp}
\Ric_f(\gamma)\le \RIC^\flat(x,y)\le  \RIC^+(x,y)\le
\Ric_f(\gamma)+\sigma(\gamma)\cdot\tan^2\Big( \sqrt{\sigma(\gamma)} d(x,y)/2\Big)\end{equation}
where $\gamma=(\gamma^a)_{a\in[-r,r]}$ denotes the (unique) geodesic connecting $x$ and $y$, $\Ric_f(\gamma)$ as introduced in \eqref{Kxy} denotes 
 the average Ricci curvature on the way  from $x$ to $y$
and with
$\sigma(\gamma)$ denotes the maximal modulus of the Riemannian curvature along this geodesic.
\end{theorem}

\begin{remark} Let $X$ be the flat torus. Then direct calculations (similar to those in Example \ref{cone}) yield that
$\RIC^\pm(x,y)=\infty$  for antipodal points $x,y$.
\end{remark}

\begin{proof}
{\bf (i) Upper estimate.} 
Given non-conjugate points $x,y\in M$, let $(\gamma^a)_{a\in[-r,r]}$  be the unit speed geodesic connecting them. 
Choose a smooth family of functions 
$\phi^a, a\in[-r,r],$ which satisfy the Hopf-Lax relation $\partial_a\phi^a=-\frac12|\nabla\phi^a|^2$ with 
$\nabla\phi^a(\gamma^a)=\dot\gamma^a$ for (some, hence all) $a\in[-r,r]$.

Then $(-2r\phi^{-r},2r\phi^r)$ is a pair of Kantorovich potentials for $\delta_x$ and $\delta_y$ and thus
\begin{eqnarray}
\frac12\partial^-_s W^2(\hat P_{s}\delta_{x}, \hat P_{s}\delta_{y})\big|_{s=0}
\nonumber
&\ge&\nonumber
\liminf_{s\searrow0}
\frac{2r}s\Big[\int  -\phi^{-r}d\hat P_{s}\delta_{x}+
\int  \phi^rd\hat P_{s}\delta_{y}-\int-\phi^{-r} d\delta_x-\int\phi^r d\delta_x\Big]\\
&=&\nonumber 2r\Big[-\Delta_f\phi^{-r}(x)+\Delta_f\phi^r(y)\big]\\
&=&\nonumber 2r\int_{-r}^r \partial_a\Big(\Delta_f\phi^a(\gamma^a)\Big)\,da\\
&=&\nonumber 2r\int_{-r}^r\Big(
-\frac12\Delta_f\big|\nabla \phi^a\big|^2+\nabla\Delta_f\phi^a\cdot\dot\gamma^a
\Big)\,da\\
&=&\nonumber-2r\int_{-r}^r\Big(
\Ric_f(\nabla\phi^a,\nabla\phi^a)+\big\|\D^2\phi^a\big\|_{2,2}^2\Big)(\gamma^a)\,da\\
&\ge&-\big(\Ric(x,y)+\sup_{a\in[-r,r]}h^2(a)\big)\cdot d^2(x,y)
\label{upper}
\end{eqnarray}
with
\begin{equation*}\label{small Hessian}
h(a):=\big\|\D^2\phi^a\big\|_{2,2}(\gamma^a)
\end{equation*} where $\D^2\phi$ denotes  the Hessian of $\phi$ and $\big\|\D^2\phi^a\big\|_{2,2}$ its Hilbert-Schmidt norm.

To estimate the latter, let us choose $\phi^0$  such that $\nabla\phi^0(\gamma^0)=\dot\gamma^0$
and $\D^2\phi^0(\gamma^0)=0$. This is indeed always possible. 
Now fix an orthonormal frame $(x_i)_{i=1,\ldots,n}$ along $\gamma$ and consider the evolution   of the Hessian $( \D_{ij}\phi^a)_{i,j=1,\ldots,n}$  under the Hopf-Lax flow:
\begin{eqnarray}\nonumber
 \partial_a\Big(\D_{ij}\phi^a(\gamma^a)\Big)
&=&
\Big[-\frac12\D_{ij}\big|\nabla \phi^a\big|^2+\nabla\D_{ij}\phi^a\cdot\dot\gamma^a\Big](\gamma^a)
\\
&=&\label{curv}
-\Big[
R(e_i,\nabla\phi^a,e_j,\nabla\phi^a)+
\sum_k\D_{ik}\phi^a\cdot \D_{jk}\phi^a\Big](\gamma^a).
\end{eqnarray}
Here 
$R(v,u,w,u)=g(\mathrm{Rm}(v,u)w, u)$ for $u,v,w\in TM$ with  $\mathrm{Rm}$ being the Riemannian curvature tensor. 

Multiplying this identity by $\D_{ij}\phi^a(\gamma^a)$, summing over $i,j$ and using the fact that the curvature is bounded by $\sigma$ yields for $h(a):=\Big(\sum_{i,j}\big(\D_{ij}\phi^a\big)^2\Big)^{1/2}(\gamma^a)$
\begin{eqnarray*}
\frac12\partial_a h^2(a)&=&\sum_{i,j}\D_{ij}\phi^a(\gamma^a)\cdot
 \partial_a\Big(\D_{ij}\phi^a(\gamma^a)\Big)\\
 &\le&\sigma\cdot h(a)+h^3(a)
\end{eqnarray*}
(cf. \cite{GaHuLa}, p. 213)
where we used the fact  that
$$\Big|\sum_{i,j,j}A_{ij} A_{jk} A_{ki}\Big| \le\Big(\sum_{i,j}A_{ij}^2\Big)^{3/2}$$
for each symmetric matrix. Dividing the previous differential inequality by $h$ yields the Riccati type  differential inequality
$$\partial_a h(a)\le \sigma+ h^2(a).$$
By assumption $h(0)=0$. Using the explicit solution for the corresponding differential \emph{equality} and Sturm's comparison principle we obtain
\begin{equation}
h(a)\le \sqrt\sigma\cdot \tan\big(\sqrt\sigma\cdot |a|\big)
\end{equation}
for all $a\in\R$ which -- together with \eqref{upper} -- proves the claim.

{\bf (ii) Lower estimate.} 
Theorem \ref{thm-prec} in the more general context of RCD-spaces will
yield the equivalence 
$\RIC^\flat=\ric$
where
\begin{eqnarray*}
\ric(x,y)&:=&\lim_{\epsilon\to0} \inf\Big\{
\frac1{W^2(\rho^0,\rho^1)}\cdot\Big[ \partial^-_a S(\rho^a)\big|_{a=1}-\partial_a^+ S(\rho^a)\big|_{a=0}
\Big]: \ \big(\rho^a\big)_{a\in[0,1]} \mbox{ $W$-geodesic},
\nonumber\\
&&\qquad\qquad S(\rho^0)<\infty, \ S(\rho^1)<\infty, \  \supp[\rho^0]\subset B_\epsilon(x), \  \supp[\rho^1]\subset B_\epsilon(y)\Big\}.
\end{eqnarray*}
Thus it remains to prove that $\ric(x,y)\ge \Ric_f(\gamma)$ for non-conjugate points $x,y$ with connecting geodesic $\gamma$.

Given non-conjugate points $x^0,x^1\in X$ with minimizing geodesic $(x^a)_{a\in[0,1]}$ and a number $\delta>0$, one can choose $\epsilon>0$ such that
$$\Ric_f(\dot\gamma^a,\dot\gamma^a)\ge 
\Ric_f(\dot x^a,\dot x^a)-\delta\qquad(\forall a)$$
for all $d$-geodesics $(\gamma^a)_{a\in[0,1]}$ with $\gamma^0\in B_\epsilon(x^0)$ and
$\gamma^1\in B_\epsilon(x^1)$.
Thus for each $W$-geodesic curve $\big(\rho^a\big)_{a\in[0,1]}$ with
$\supp[\rho^0]\subset B_\epsilon(x), \supp[\rho^1]\subset B_\epsilon(y)$
by standard calculations for first and second derivatives of $a\mapsto S(\rho^a_t)$
(cf. \cite{Sturm04}, \cite{AGS14Duke}, \cite{KoSt})
\begin{eqnarray*}
 \partial_a^- S(\rho^a_t)\big|_{a=1}-\partial_a^+ S(\rho^a_t)\big|_{a=0}
 &\ge&\int_\Gamma\int_0^1 \Ric_f(\dot\gamma^a,\dot\gamma^a)\,da\,d\Lambda(\gamma)\\
 &\ge&\int_0^1 \Ric_f(\dot x^a,\dot x^a)\,da-\delta\\
 &\ge& \Ric_f(x^\cdot)\cdot d^2(x^0,x^1)-\delta\\
 &\ge& \Ric_f(x^\cdot)\cdot W^2(\rho^0,\rho^1)-2\delta.
 \end{eqnarray*}
 Here
  $\Lambda_t$ denotes the probability measure on the space $\Gamma$ of all geodesics $\gamma:[0,1]\to X$ such that
 $\rho_t^a=(\pi^a)_\sharp \Lambda_t$
 where $\pi^a:\gamma\mapsto\gamma^a$ is the evaluation map.
 Since $\delta>0$ was arbitrary this proves $\ric(x^0,x^1)\ge \Ric_f(x^\cdot)$ which is the claim.
\end{proof}

\begin{remark}
A careful look on the previous argument allows to replace the maximum of the modulus of the curvature (along the minimizing geodesic from $x$ to $y$) in the definition of  $\sigma(\gamma)$ by
$$\sigma(\gamma)=\max_{a} \Big(\sum_{i,j=1}^n \big|R (\dot\gamma, e_i,\dot\gamma, e_j)\big|^2\Big)^{1/2}(\gamma^a)
=\max_{a} \Big(\sum_{i=1}^n \big|\mathrm{Rm} (\dot\gamma, e_i)\dot\gamma\big|^2\Big)^{1/2}(\gamma^a).$$
\end{remark}

\subsection{Speed and Action}

For various applications we have to tighten up the asymptotic estimates in the previous Theorem to estimates which hold locally uniformly in time.

\begin{theorem}\label{uni-lower}
For each compact subset $M_0\subset M$ and each $\delta>0$ there exist $t_0>0$ and $\epsilon>0$ such that for all $x,y\in M_0$ with $d(x,y)\le \epsilon$
\begin{equation}
e^{-(\Ric_f(x,y)+\delta)t}\cdot d(x,y)\le W\Big(\hat P_t\delta_x,\hat P_t\delta_y\Big)\le e^{(-\Ric_f(x,y)+\delta)t}\cdot d(x,y)
\end{equation}
with $\Ric_f(x,y):= \Ric_f(\gamma^\cdot)$
as in Theorem \ref{thm-sharp} being the average of the Ricci curvature along the  geodesic $\gamma$ connecting $x$ and $y$.
\end{theorem}

\begin{proof} (i)
The upper estimate will follow 
from the fact that $\Ric_f(x,y)\le\ric(x,y)$ (proof of ``lower estimate'' in previous Theorem) and 
from the estimate \eqref{ppp} in the proof of Theorem \ref{thm-prec} below.

(ii) For the lower estimate, we have to improve our ``upper estimate'' in  Theorem \ref{thm-sharp}. Let us first prove that for  each compact subset $M_0\subset M$ and each $\delta>0$ there exist constants $C$ and $\epsilon>0$ such that for each unit speed geodesic $(\gamma^a)_{a\in [-r,r]}$ of length $2r\le \epsilon$ in $M_0$
there exists a
 family of functions 
$\phi^a, a\in[-r,r],$ which satisfy the Hopf-Lax relation $\partial_a\phi^a=-\frac12|\nabla\phi^a|^2$ with 
$\nabla\phi^a(\gamma^a)=\dot\gamma^a$ for  $a\in[-r,r]$ and 
\begin{equation}
\label{third order}
|\mathcal D^3 \phi|(a)\le C\qquad \mbox{for }a\in\{-r,r\}.\end{equation}

Indeed, as before we  fix an orthonormal frame $(x_i)_{i=1,\ldots,n}$ along $\gamma$ with $e_1=\dot\gamma$ and
put $\phi^0=\frac12 x_1^2$. For each $i,j,k$ we consider the evolution of the third derivatives $( \D_{ijk}\phi^a)_{i,j=1,\ldots,n}$  under the Hopf-Lax flow which amounts to a coupled system of nonlinear PDEs with smooth coefficients (depending on $\phi, \D\phi, \D^2\phi$, the Riemannian curvature tensor and its first derivative).
These coefficients are uniformly bounded on the compact set $M_0$.
Since the initial condition of theses coupled system at $a=0$ is $ \D_{ijk}\phi^0=0$, short time existence implies that there exist $\epsilon>0$ and $C$ such that $|\mathcal D^3 \phi|(a)\le Cr$ for $a\in [-r,r]$ provided  $r\le \epsilon/2$.

(iii) With the uniform estimate \eqref{third order} at hands, we will improve the upper estimate in  Theorem \ref{thm-sharp}. Note that for all $s>0$
\begin{eqnarray*}
\frac1{2s}\Big[ W^2(\hat P_{s}\delta_{x}, \hat P_{s}\delta_{y})-d^2(x,y)\Big]
\nonumber
&\ge&\nonumber
\frac{2r}s\Big[\int  -\phi^{-r}d\hat P_{s}\delta_{x}+
\int  \phi^rd\hat P_{s}\delta_{y}-\int-\phi^{-r} d\delta_x-\int\phi^r d\delta_x\Big]\\
&\ge&\nonumber 2r\Big[-\Delta_f\phi^{-r}(x)+\Delta_f\phi^r(y)-Cr\cdot s\big]\\
&\ge& -\big[\Ric_f(x,y)+h^2+C\cdot s/2\big]\cdot d^2(x,y)
\end{eqnarray*}
uniformly in $x,y\in M_0$ provided $d(x,y)\le\epsilon$ where $x=\gamma^{-r}$ and $y=\gamma^r$.
Here $$h=\sup_{x,y\in M_0}\sqrt{\sigma(x,y)}\cdot \tan\big( \sqrt{\sigma(x,y)}\cdot\epsilon\big)$$
with $\sigma(x,y)$ as before denoting the maximum modulus of the Riemannian curvature on the way from $x$ to $y$.
Thus the claim follows: for sufficiently small $\epsilon$ and  $s_0$ and all $s\in (0,s_0)$
\begin{equation*}
W\Big(\hat P_s\delta_x,\hat P_s\delta_y\Big)\ge e^{-(\Ric_f(x,y)+\delta)s}\cdot d(x,y).
\end{equation*}
\end{proof}

\begin{theorem}\label{ptw-sharp}
For each $v=(x,\xi)\in TM$
\begin{equation}\label{geo-sharp} 
-\partial_t\log\speed(\hat P_t\delta_{\gamma^a})\big|_{a=0,t=0}=\Ric_f(v,v)/|v|^2
\end{equation}
and
\begin{equation}\label{geo-sharpII} 
-\partial_t\speed(\hat P_t\delta_{\gamma^a})\big|_{a=0,t=0}=\Ric_f(v,v)/|v|
\end{equation}
where $\gamma^a=\exp_x(a\xi)$ for $a\in[-r,r]$ with some $r>0$ and where $$\speed(\mu^a)=\limsup_{b\to a}\frac1{|b-a|}W(\mu^a,\mu^b)
$$
denotes the metric speed of a curve $(\mu^a)_{a\in[-r,r]}$ in $(\Pz(X),W)$.
\end{theorem}

\begin{proof} Given $v\in TM$ and $\delta>0$, put $\kappa=\Ric_f(v,v)/|v|^2$. Then for some $r>0$ and all $a\in [-r,r]$
$$\kappa+\delta\ge \Ric_f(\dot\gamma^a,\dot\gamma^a)/|\dot\gamma^a|^2
\ge \kappa-\delta.$$
Therefore, for sufficiently small $t_0$ and all $t\le t_0$
$$e^{-(\kappa+2\delta)t}\cdot d(\gamma^a,\gamma^b)\le
W(\mu_t^a,\mu_t^b)\le e^{(-\kappa+2\delta)t}\cdot d(\gamma^a,\gamma^b)$$
where $\mu_t^a= \hat P_t\delta_{\gamma^a}$, and thus
$$e^{-(\kappa+2\delta)t}\cdot \speed(\gamma^a)\le \speed(\mu_t^a)\le e^{(-\kappa+2\delta)t}\cdot \speed(\gamma^a).$$
This implies
$$\kappa+2\delta\ge -\partial_t^\pm\log \speed(\mu_t^a)\big|_{t=0}\ge \kappa-2\delta$$
and thus (since $\delta>0$ was arbitrary)
$$ -\partial_t\log \speed(\gamma^a)\big|_{t=0}= \kappa=\Ric_f(v,v)/|v|^2.$$
The result for $\partial_t \speed(\mu_t^a)$ follows by chain rule.
\end{proof}

\begin{corollary}\label{thm-curves-sharp}
For each $\mathcal C^1$-curve $(\gamma^a)_{a\in[0,1]}$  in $M$ and each $p\in[1,\infty)$
\begin{equation}\label{curves-sharp-II} 
-\frac1p\partial_t\action_p(\hat P_t\delta_{\gamma^\cdot})\big|_{t=0}= \Ric^{(p)}_f(\gamma)
\end{equation}
with $
\Ric^{(p)}_f(\gamma):= \int_{0}^1 \Ric_f(\dot\gamma^a, \dot\gamma^a)\, |\dot\gamma^a|^{p-2}\,da
$. In particular, 
\begin{equation}\label{curves-sharp-IIa} 
-\frac12\partial_t\action_2(\hat P_t\delta_{\gamma^\cdot})\big|_{t=0}=  \int_{0}^1 \Ric_f(\dot\gamma^a, \dot\gamma^a)\,da.
\end{equation}

Moreover, for each $\mathcal C^1$-curve parametrized with constant speed
\begin{equation}\label{curves-sharp-II} 
-\frac1p\partial_t\log\action_p(\hat P_t\delta_{\gamma^\cdot})\big|_{t=0}= \Ric_f(\gamma)
\end{equation}
with $\Ric_f(\gamma):=
\Ric^{(0)}_f(\gamma)= \int_{0}^1 \Ric_f(\dot\gamma^a, \dot\gamma^a)\, |\dot\gamma^a|^{-2}\,da
$ being the mean value of   $\Ric_f(\dot\gamma^a, \dot\gamma^a)/ |\dot\gamma^a|^2$,
the normalized Ricci curvature along the curve.
\end{corollary}

\begin{proof} 
The results concerning $\partial_t\log\action_p(\,.\,)$ for curves with constant speed are immediate consequences of the previous results concerning $\partial_t\action_p(\,.\,)$ since in this case $\action_p(\gamma^\cdot)= |\dot\gamma^a|^{p}$.

For the  $\le$-estimate in the first assertion, we apply  the previous Theorem.
Given $\delta>0$, $M_0$, $t_0$ and $\epsilon>0$ as above,  for every rectifiable curve $\gamma$ within $M_0$ and for all sufficiently fine,  finite partitions $0=a_0<a_1<\ldots<a_n=1$ and all $t\le t_0$ we have 
$$W\big(\mu_t^{a_i}, \mu_t^{a_{i-1}}\big)\ge e^{-\big(\Ric_f(\gamma^{a_{i-1}},\gamma^{a_{i}})+\delta\big)t}\cdot d\big(\gamma^{a_i}, \gamma^{a_{i-1}}\big).$$
Thus
\begin{eqnarray*}
\action_p(\mu_t^\cdot)&\ge&
\sum_i \frac1{|a_i-a_{i-1}|^{p-1}}W^p\big(\mu_t^{a_i}, \mu_t^{a_{i-1}}\big)\\
&\ge&
\sum_i \frac1{|a_i-a_{i-1}|^{p-1}}\Big[1-\Big( \Ric_f(\gamma^{a_{i-1}},\gamma^{a_{i}}) +\delta\Big)pt\Big]\cdot d^p\big(\gamma^{a_i}, \gamma^{a_{i-1}}\big)\\
&\ge&
\sum_i \frac1{|a_i-a_{i-1}|^{p-1}}\Big[1-\Big(\Ric_f(\dot\gamma^{a_i}, \dot\gamma^{a_i})/|\dot\gamma^{a_i}|^2 +2\delta\Big)pt\Big]\cdot d^p\big(\gamma^{a_i}, \gamma^{a_{i-1}}\big)
\end{eqnarray*}
since we assumed the curve to be  $\mathcal C^1$.
Passing to finer and finer partitions this implies
\begin{eqnarray*}
\action_p(\mu_t^\cdot)&\ge&
\action_p(\gamma^\cdot)\cdot [1-2\delta pt]  -pt\cdot \int_0^1 \Ric_f(\dot\gamma^a, \dot\gamma^a)\, |\dot\gamma^a|^{p-2}\,da\\
&=&
\action_p(\gamma^\cdot)\cdot [1-2\delta pt]-\Ric_f^{(p)}(\gamma)
\end{eqnarray*}
Since $\delta>0$ was arbitrary this proves the the upper estimate. Note that in the constant speed case, the last inequality can be rewritten as
\begin{eqnarray*}
\action_p(\mu_t^\cdot)&\ge&
\action_p(\gamma^\cdot)\cdot [1-2\delta pt+\Ric_f^{(0)}(\gamma)]
\end{eqnarray*}

For the $\ge$-estimate we use the fact that $\action_p(\gamma)=\int_0^1|\dot\gamma^a|^p da$ for each absolutely continuous curve.
Thus
\begin{eqnarray*}
\limsup_{t\to0}
\frac1t \Big[\action_p\big(\hat P_t\delta_{\gamma^\cdot}\big)-\action_p\big(\gamma^\cdot\big)\Big]
&=&\limsup_{t\to0}\int_0^1
\frac1t \Big[\speed\big(\hat P_t\delta_{\gamma^a}\big)^p-\speed\big(\gamma^a\big)^p\Big]da\\
&\le&\int_0^1\limsup_{t\to0}
\frac1t \Big[\speed\big(\hat P_t\delta_{\gamma^a}\big)^p-\speed\big(\gamma^a\big)^p\Big]da\\
&\le&p\int_0^1\Ric_f(\dot\gamma^a, \dot\gamma^a)\cdot \speed\big(\gamma^a\big)^{p-2}da
\end{eqnarray*}
where application of Fatou's lemma in the first inequality is justified since 
$$\speed\big(\hat P_t\delta_{\gamma^a}\big)\le e^{-K't}\speed\big(\gamma^a\big)$$
thanks to our a priori assumption $\Ric_f\ge K'$. The last inequality in the previous chain follows from  Theorem \ref{ptw-sharp} from above.
\end{proof}

\subsection{Upper Ricci Bounds}

\begin{theorem} For any weighted Riemannian manifold  $(M, g,f)$ satisfying some RCD$(K',N)$-condition and any number $\kappa\in\R$  the following assertions are equivalent:
\begin{itemize}
\item[(i)] $\Ric_f\le\kappa\cdot g$ on $TM$

\item[(ii)] For every $x\in X$ 
$$\limsup_{y,z\to x}\RIC^+(y,z)\le \kappa$$ 
\item[(iii)] For every $x\in X$ 
$$\limsup_{y,z\to x}\RIC^\flat(y,z)\le \kappa$$

\item[(iv)] For every geodesic $a\mapsto \gamma^a$ in $M$ there exists $p\in[1,\infty)$ such that
\begin{equation}\label{act-upper}
-\frac1p\partial_t^-\log\action_p(\hat P_t\delta_{\gamma^\cdot})\big|_{t=0}\le\kappa.
\end{equation}

\item[(v)] For every 
curve $a\mapsto \gamma^a$ in $M$, every $p\in[1,\infty)$ and every $\delta>0$ there exists $t_0>0$ such that for  all $t\in (0,t_0)$
\begin{equation}
\action_p(\hat P_t\delta_{\gamma^\cdot})\ge e^{-(\kappa+\delta) t}\cdot \action_p(\gamma^\cdot)
\end{equation}
\end{itemize}
\end{theorem}

According to various implications proven so far, there are plenty of further assertions which are equivalent to the previous ones.
For instance,
``$\limsup_{y,z\to x}\RIC^-(y,z)\le \kappa$ for all $x$''
which is weaker than (ii) and stronger than (iii)  or ``\eqref{act-upper} for every geodesic $\gamma^\cdot$ and \emph{every} $p\in[1,\infty)$'' 
which is weaker than (v) and stronger than (iv).

\begin{proof} It remains to prove (i) $\Rightarrow$ (v). Given the curve $\gamma$ and $\delta>0$, according to Theorem \ref{uni-lower}
the upper Ricci bound (i) implies that there exists $t_0>0$ such that
\begin{equation*}
W\Big(\hat P_t\delta_x,\hat P_t\delta_y\Big)\ge e^{-(\kappa+\delta) t}\cdot d(x,y)
\end{equation*}
uniformly for all $x,y$ in the range of $\gamma$ and all $t\in (0,t_0)$. Therefore, for each partition $0=a_0<a_1<\ldots<a_k=1$ we have
\begin{equation*}
\sum_i  \frac1{|a_{i-1}-a_i|^{p-1}}W\Big(\hat P_t\delta_{\gamma^{a_{i-1}}},\hat P_t\delta_{\gamma^{a_{i}}}\Big)^p\ge e^{-(\kappa+\delta) pt}\cdot \sum_i  \frac1{|a_{i-1}-a_i|^{p-1}} d({\gamma^{a_{i-1}}},{\gamma^{a_{i}}})^p
\end{equation*}
and thus
$\action_p(\hat P_t\delta_{\gamma^\cdot})\ge e^{-(\kappa+\delta) pt}\cdot \action_p(\gamma^\cdot)$.
\end{proof}

\section{Ricci Bounds in Terms of Convexity/Concavity of the Entropy}
Let us briefly recall the Lott-Sturm-Villani definition of synthetic lower bounds for the Ricci curvature \cite{SturmActa1, LV1}. 
The synthetic lower bound CD$(K,\infty)$  is defined as the weak $K$-convexity of the Boltzmann entropy $S$ on the geodesic space $(\Pz(X),W)$. For infinitesimally Hilbertian spaces, this can be rephrased as follows \cite{Sturm-srf}: for each geodesic  $\big(\rho^a\big)_{a\in[0,1]}$ with $S(\rho^0), S(\rho^1)<\infty$, the function $a\mapsto S(\rho^a)$ is upper semicontinuous  on
$[0, 1]$, absolutely continuous on $(0, 1)$, and satisfies
\begin{eqnarray}\label{ent-end}\frac1{W^2(\rho^0,\rho^1)}\cdot\Big[ \partial^-_a S(\rho^a)\big|_{a=1}-\partial_a^+ S(\rho^a)\big|_{a=0}
\Big]\ge K
\end{eqnarray}
where $$ \partial^-_a S(\rho^a)\big|_{a=1}=\liminf_{a\nearrow1}\frac1{1-a}\big[S(\rho^1)-S(\rho^{a}\big], \qquad\partial^+_a S(\rho^a)\big|_{a=0}=\limsup_{a\searrow0}\frac1a\big[S(\rho^a)-S(\rho^{0}\big].$$
The synthetic upper Ricci bounds to be presented below will reverse this inequality: instead of $K$-convexity we will request $\kappa$-concavity.
However, in the case of upper bounds, this $\kappa$-concavity property will be requested only for
$W$-geodesics with endpoints supported in arbitrarily small neighborhoods of given points.

\begin{definition}
Given a mm-space $(X,d,m)$ be define the function
$\ric: X\times X\to [-\infty,\infty]$ by
\begin{eqnarray}
\ric(x,y)&:=&\lim_{\epsilon\to0}\  \inf\Big\{
\frac1{W^2(\rho^0,\rho^1)}\cdot\Big[ \partial^-_a S(\rho^a)\big|_{a=1}-\partial_a^+ S(\rho^a)\big|_{a=0}
\Big]: \ \big(\rho^a\big)_{a\in[0,1]} \mbox{ $W$-geodesic},
\nonumber\\
&&\qquad\qquad S(\rho^0)<\infty, \ S(\rho^1)<\infty, \  \supp[\rho^0]\subset B_\epsilon(x), \  \supp[\rho^1]\subset B_\epsilon(y)\Big\}.
\label{ric}
\end{eqnarray}
\end{definition}

\begin{theorem}\label{thm-prec} For each mm-space $(X,d,m)$ satisfying an RCD$(-K,N)$-condition (for some finite numbers $K,N$) and for each pair of points $x,y\in X$
$$ \RIC^\flat(x,y)= \ric(x,y).$$
\end{theorem}

\begin{proof}
The $\le$-inequality will be proven in a slightly more general context as the next Theorem. 

To prove the $\ge$-inequality, fix $x,y\in X$, $\delta>0$ and put $\kappa=\ric(x,y)$.
For $\epsilon>0$ small (to be made more precise in the sequel), let $\mu_0,\nu_0$ be given with 
$\supp[\mu_0]\subset B_{\epsilon/2}(x), \supp[\nu_0]\subset B_{\epsilon/2}(y)$.

For $t>0$  let
 $(\rho_t^a)_{a\in[0,1]}$ denote the $W$-geodesic connecting  $\mu_t=\hat P_t\mu=u_t\,m_0$ and $\nu_t=\hat P_t\nu_0=v_t\,m$. Denote the associated conjugate Kantorovich potentials by $\phi_t$ and $\psi_t$. Then by standard calculations in RCD$(-K,\infty)$-spaces (due to \cite{AGS14Duke}, cf. Chapter 4)
 \begin{eqnarray}\nonumber
 -\frac12 \partial_s^+ W^2(\mu_s,\nu_s)\big|_{s=t-}
 &\ge&\nonumber
 \E(\phi_t,u_t)+\E(\psi_t,v_t)\\
 &\ge &
 \partial_a^- S(\rho^a_t)\big|_{a=1}-\partial_a^+ S(\rho^a_t)\big|_{a=0}.\label{s1}
 \end{eqnarray}
 
 Let $\Lambda_t$ denote the probability measure on the space $\Gamma$ of all geodesics $\gamma:[0,1]\to X$ such that
 $\rho_t^a=(\pi^a)_\sharp \Lambda_t$
 where $\pi^a:\gamma\mapsto\gamma^a$ is the evaluation map. Put
 $\Gamma_\epsilon=\big\{\gamma\in\Gamma: \gamma^0\in B_\epsilon(x), \gamma^1\in B_\epsilon(y)\big\}$,
 $\lambda_{t,\epsilon}=\Lambda_t(\Gamma_\epsilon)$ and
 $$\Lambda_{t,\epsilon}\big(\,  . \,\big)=\frac1{\lambda_{t,\epsilon}}\cdot \Lambda_t\big(\, .  \,\cap \Gamma_\epsilon\big)\quad\mbox{
 as well as}\quad
 \overline\Lambda_{t,\epsilon}\big(\, . \, \big)=\frac1{1-\lambda_{t,\epsilon}}\cdot \Lambda_t\big(\, . \, \cap \complement\Gamma_\epsilon\big).$$
 Moreover, put 
  $\rho_{t,\epsilon}^a=(\pi^a)_\sharp \Lambda_{t,\epsilon}$ and 
 $\overline\rho_{t,\epsilon}^a=(\pi^a)_\sharp \overline\Lambda_{t,\epsilon}$.
 Then
 $\Lambda_t=\lambda_{t,\epsilon}\cdot \Lambda_{t,\epsilon}+(1-\lambda_{t,\epsilon})\cdot \overline\Lambda_{t,\epsilon}$ as well as
  $\rho^a_t=\lambda_{t,\epsilon}\cdot \rho^a_{t,\epsilon}+(1-\lambda_{t,\epsilon})\cdot \overline\rho^a_{t,\epsilon}$.
 Optimality of $(\rho_t^a)_{a\in[0,1]}$ implies
  $$W^2\big(\rho^0_t,\rho^1_t\big)=\lambda_{t,\epsilon}\cdot W^2\big(\rho^0_{t,\epsilon},\rho^1_{t,\epsilon}\big)
  +(1-\lambda_{t,\epsilon})\cdot W^2\big(\overline\rho^0_{t,\epsilon},\overline\rho^1_{t,\epsilon}\big).$$
  Furthermore, the essential non-branching property of RCD$(-K,\infty)$-spaces implies that for all $a\in (0,1)$ the measures $\rho^a_{t,\epsilon}$ and $\overline\rho^a_{t,\epsilon}$ are supported by disjoint sets. Therefore
  $$S(\rho^a)=\lambda_{t,\epsilon}\cdot S(\rho^a_{t,\epsilon})+(1-\lambda_{t,\epsilon})\cdot S(\overline\rho^a_{t,\epsilon})+\lambda_{t,\epsilon}\log\lambda_{t,\epsilon}+(1-\lambda_{t,\epsilon})\log(1-\lambda_{t,\epsilon})$$
  whereas for $a=0$ and $a=1$
   $$S(\rho^a)\ge\lambda_{t,\epsilon}\cdot S(\rho^a_{t,\epsilon})+(1-\lambda_{t,\epsilon})\cdot S(\overline\rho^a_{t,\epsilon})+\lambda_{t,\epsilon}\log\lambda_{t,\epsilon}+(1-\lambda_{t,\epsilon})\log(1-\lambda_{t,\epsilon}).$$
Thus
\begin{eqnarray}\nonumber
\partial_a^- S(\rho^a_t)\big|_{a=1}-\partial_a^+ S(\rho^a_t)\big|_{a=0}
&\ge&
\lambda_{t,\epsilon}\Big[
\partial_a^- S(\rho^a_{t,\epsilon})\big|_{a=1}-\partial_a^+ S(\rho^a_{t,\epsilon})\big|_{a=0}\Big]\\
&&
+(1-\lambda_{t,\epsilon})\Big[
\partial_a^- S(\overline\rho^a_{t,\epsilon})\big|_{a=1}-\partial_a^+ S(\overline\rho^a_{t,\epsilon})\big|_{a=0}\Big].\label{s2}
 \end{eqnarray}
 
 By construction, the $W$-geodesic $(\rho_{t,\epsilon}^a)_{a\in[0,1]}$ is supported by $d$-geodesics $(\gamma^a)_{a\in[0,1]}$ emanating  in $B_\epsilon(x)$ and terminating 
  in $B_\epsilon(y)$. Thus choosing $\epsilon>0$ small enough, by the very definition of $\kappa=\ric(x,y)$ we can ensure that
 \begin{equation}\label{s3}
 \partial_a^- S(\rho^a_{t,\epsilon})\big|_{a=1}-\partial_a^+ S(\rho^a_{t,\epsilon})\big|_{a=0}\ge
 (\kappa-\delta)\cdot W^2\big(\rho^0_{t,\epsilon},\rho^1_{t,\epsilon}\big).
 \end{equation}

 On the other hand, for the geodesic $(\overline\rho_{t,\epsilon}^a)_{a\in[0,1]}$ we make use of the a priori assumption CD$(-K,\infty)$ which yields
  \begin{equation}\label{s4}
 \partial_a^- S(\overline\rho^a_{t,\epsilon})\big|_{a=1}-\partial_a^+ S(\overline\rho^a_{t,\epsilon})\big|_{a=0}\ge
 -K\cdot W^2\big(\overline\rho^0_{t,\epsilon},\overline\rho^1_{t,\epsilon}\big).
 \end{equation}
Summing up the estimates \eqref{s1} - \eqref{s4} we thus obtain
\begin{eqnarray}\nonumber
 -\frac12 \partial^+_s W^2(\mu_s,\mu_s)\big|_{s=t-}
 &\ge&\nonumber
\lambda_{t,\epsilon}
 (\kappa-\delta)\cdot W^2\big(\rho^0_{t,\epsilon},\rho^1_{t,\epsilon}\big)
 -(1-\lambda_{t,\epsilon})
 K\cdot W^2\big(\overline\rho^0_{t,\epsilon},\overline\rho^1_{t,\epsilon}\big)\\
 &=&
 (\kappa-\delta)\cdot W^2\big(\mu_{t},\nu_{t}\big)
 -(1-\lambda_{t,\epsilon})
 (K+\kappa-\delta)\cdot W^2\big(\overline\rho^0_{t,\epsilon},\overline\rho^1_{t,\epsilon}\big).
 \end{eqnarray}
 It remains to control the second term in the last line.
 
 On RCD$(-K,N)$-spaces (due to controlled growth of doubling and Poincar\'e constants) we have  upper Gaussian estimates for the heat kernel and for the volume growth which allow to estimate
 \begin{eqnarray*}
 1-\lambda_{t,\epsilon}&\le&\sup_{x'\in B_{\epsilon/2}(x)}
 \int_{\complement B_\epsilon(x)}p_t(x',z)\,dm(z)+
 \sup_{y'\in B_{\epsilon/2}(y)}
 \int_{\complement B_\epsilon(y)}p_t(y',z)\,dm(z)\\
 &\le&
 \frac C{m(B_{\sqrt t}(x))}\int_\epsilon^\infty e^{-\frac{r^2}{Ct}}\cdot |\partial_r m(B_r(x)|\,dr\\
 &&\qquad\quad+
 \frac C{m(B_{\sqrt t}(y))}\int_\epsilon^\infty e^{-\frac{r^2}{Ct}}\cdot |\partial_r m(B_r(y)|\,dr\\
 &\le&
 \frac{2C}{\sqrt t}\int_\epsilon^\infty e^{-\frac{r^2}{Ct}+\sqrt{|K|N}r}\cdot \Big(\frac {r}{\sqrt t}\Big)^{N-1}\,dr\
 \le\ C'e^{-\epsilon^2/(C't)}
 \end{eqnarray*}
 uniformly in $x,y\in X$ and $t\in(0,1]$ (see e.g. \cite{Sturm-DirII}, \cite{LierlSa}) as well as
 \begin{eqnarray*}
 \inf_{x'\in B_{\epsilon/2}(x)}\int_{\complement B_\epsilon(x)}p_t(x',z)\,dm(z)
 &\ge&
 \frac 1{C\, m(B_{\sqrt t}(x))}\int_\epsilon^{2\epsilon} e^{-C\frac{r^2}{t}}\cdot |\partial_r m(B_r(x)|\,dr.
 \end{eqnarray*}
Moreover, 
  \begin{eqnarray*}
  \limsup_{t\searrow 0} W\big(\overline\rho^0_{t,\epsilon},\overline\rho^1_{t,\epsilon}\big)
 &\le&
  \limsup_{t\searrow 0} W\big(\overline\rho^0_{t,\epsilon},\delta_x\big)+d(x,y)+
   \limsup_{t\searrow 0} W\big(\delta_y,\overline\rho^1_{t,\epsilon}\big)
  \end{eqnarray*}
  and
  \begin{eqnarray*}
  W^2\big(\overline\rho^0_{t,\epsilon},\delta_x\big)
  &\le&
  \frac1{\int\int_{\complement B_\epsilon(x)}p_t(x',z)\,dm(z)d\mu_0(x')}
  \int\int_{\complement B_\epsilon(x)}d^2(x',z)p_t(x,z)\,dm(z)d\mu_0(x')\\
  &\le&(1+\epsilon/2)^2+
  \frac{\sup_{x'\in B_{\epsilon/2}(x)}\int_{\complement B_{1}(x)}d^2(x,z)p_t(x',z)\,dm(z)}{\inf_{x'\in B_{\epsilon/2}(x)} \int_{\complement B_\epsilon(x)}p_t(x',z)\,dm(z)}
  \\
   &\le&2+C^2\frac{
   \int_{1}^\infty r^2\cdot e^{-\frac{r^2}{Ct}}e^{\sqrt{|K|N}r}\cdot |\partial_r m(B_r(x)|\,dr
   }{
   \int_\epsilon^{2\epsilon} e^{-C\frac{r^2}{t}}\cdot |\partial_r m(B_r(x)|\,dr
   }\\
   &\le&2+C'\frac{
   \int_1^\infty r^{N+1}\cdot e^{-\frac{r^2}{Ct}}e^{\sqrt{|K|N}r}\,dr
   }{\epsilon^{N-1}\cdot
   \int_\epsilon^{2\epsilon} e^{-C\frac{r^2}{t}}\,dr
   }\ \le \ 2+\overline C \frac{\sqrt t}{\epsilon^N}
   \end{eqnarray*}
   for some constant $\overline C$ and all $t\le1$
   provided $\epsilon\le \frac1{ 2 C}$.
 
   Thus
   $$(1-\lambda_{t,\epsilon})
 \cdot W^2\big(\overline\rho^0_{t,\epsilon},\overline\rho^1_{t,\epsilon}\big)\to0$$
 for $t\to0$.
 Finally, note that   
 $\liminf_{t\to0}W^2\big(\rho^0_{t},\rho^1_{t}\big)>0$. Hence, for some $t_\epsilon>0$ and all  $t\in (0,t_\epsilon)$
 \begin{eqnarray*}
 -\frac12 \partial^+_s W^2(\mu_s,\nu_s)\big|_{s=t-}
 &\ge&
 (\kappa-2\delta)\cdot W^2\big(\mu_{t},\nu_{t}\big).
 \end{eqnarray*}
 Integrating this w.r.t.\ $t$ yields
 \begin{eqnarray}\label{ppp}W^2(\mu_t,\nu_t)\le e^{-2(\kappa-2\delta)(t-s)}
 W^2(\mu_s,\nu_s)\end{eqnarray}
 for all $0<s<t<t_\epsilon$. Letting first $s$ go to 0 and then $t$ go to 0 yields \eqref{ppp} (with right instead of left derivative) for $t=0$. More precisely,
 \begin{eqnarray*}
 -\frac12 \partial_t^+ W^2(\mu_t,\nu_t)\big|_{t=0}
 &\ge&
 (\kappa-2\delta)\cdot W^2\big(\mu_{t},\nu_{t}\big).
 \end{eqnarray*}
 Since $\delta>0$ can be chosen arbitrarily small, this proves 
 \begin{eqnarray*}
 -\frac12 \partial_t^+ W^2(\mu_t,\nu_t)\big|_{t=0}
 &\ge&
 \kappa\cdot W^2\big(\mu_{t},\nu_{t}\big).
 \end{eqnarray*}
 Now letting first  $\epsilon\to0$ and then $\delta\to0$,
this yields the claim.
\end{proof}

\begin{theorem}\label{thm-prec-3} For each mm-space $(X,d,m)$ satisfying some RCD$(-K,\infty)$-condition and for each pair of points $x,y\in X$
$$ \RIC^\flat(x,y)\le \ric(x,y).$$
\end{theorem}

\begin{proof}
Fix $x,y\in X$ and $\delta>0$.
For $\epsilon>0$ small (to be specified  in the sequel), let $\mu^0,\mu^1$ be given with 
$\supp[\mu^0]\subset B_{\epsilon/2}(x), \supp[\mu^1]\subset B_{\epsilon/2}(y)$.
Let $(\mu^a)_{a\in[0,1]}$ denote the connecting geodesic and put
$$\mu^a_t=\hat P_t\mu^a.$$

For $a,b$ small enough, $\supp[\mu^a]\subset B_{\epsilon}(x), \supp[\mu^{1-b}]\subset B_{\epsilon}(y)$.
Thus for $t$ sufficiently small we can estimate in terms of $\kappa:=\RIC^\flat(x,y)$
\begin{equation}\label{mid}
W^2(\mu_t^a,\mu_t^{1-b})\le e^{2(\delta-\kappa)t}\cdot W^2(\mu^a,\mu^{1-b})
\end{equation}
(provided we had chosen $\epsilon$ approriately).

Moreover, applying the EVI$(-K,\infty)$-property of the dual heat flow yields
\begin{equation}\label{evi-0}
e^{-Kt}\,W^2(\mu^0,\mu^a_t)-W^2(\mu^0,\mu^a)\le 2\frac{1-e^{-Kt}}K \cdot\big[S(\mu^0)-S(\mu^a_t)\big]
\end{equation}
as well as
\begin{equation}\label{evi-1}
e^{-Kt}\,W^2(\mu^1,\mu^{1-b}_t)-W^2(\mu^1,\mu^{1-b})\le 2\frac{1-e^{-Kt}}K \cdot\big[S(\mu^1)-S(\mu^{1-b}_t)\big].
\end{equation}
Multiplying inequality \eqref{evi-0} by $\frac1{2at}$, inequality \eqref{evi-1} by $\frac1{2bt}$, inequality \eqref{mid} by $\frac1{2(1-a-b)t}$,  adding up the three resulting inequalities yields, and passing to the limit $t\to0$ yields
\begin{eqnarray*}
\big[(1-a-b)(\kappa-\delta)-(a+b)\frac K2\big]\cdot W^2(\mu^0,\mu^1)
\le \frac1a\big[S(\mu^0)-S(\mu^a_t)\big] +\frac1b\big[S(\mu^1) -S(\mu^{1-b}_t)\big].
\end{eqnarray*}
Thus for $a,b\to 0$ we obtain
$$\kappa-\delta\le\frac1{W^2(\mu^0,\mu^1)}\cdot\Big[ \partial^-_a S(\rho^a_t)\big|_{a=1}-\partial^+_a S(\rho^a_t)\big|_{a=0}
\Big]$$
which is the claim (since $\delta>0$ was arbitrary).
\end{proof}

\section{A Stable Upper Bound}
Our previous characterization of upper Ricci bounds in terms of the function $\ric$ was based on a partial reversal of the characterization of lower Ricci bounds \eqref{ent-end}.
Now we will formulate an alternate characterization of upper Ricci bounds
based on a partial reversal of the characterization of lower Ricci bounds \eqref{ent-mid}.
Recall that the latter characterization is stable under convergence of  mm-spaces provided the mid-point convexity of the entropy is properly formulated: given the endpoints, the midpoint (or the connecting geodesic) can be chosen freely.

 To obtain stability for the synthetic upper Ricci bound, the crucial trick now is to formulate the mid-point concavity of the entropy  in such a way that the midpoint and one endpoint will (essentially) be given and the other endpoint can be chosen freely.
 
 The subsequent results will hold in the general context of metric measure spaces. Neither infinitesimal Hilbertianity nor any kind of curvature-dimension condition is requested. To simplify the subsequent presentation, we assume that the measure $m$ of a mm-space always has full support.

\begin{definition}
We say that  the mm-space $(X,d,m)$ has  robust synthetic upper Ricci bound $K$ if there exists an upper semicontinuous function $\overline K:(0,\infty)\to(-\infty,\infty]$ with $\lim_{r\to 0}\overline K(r)=K$  such that for all $r>0$ with $\overline K(r)<\infty$,  all $x,y\in X$    with $d(x,y)=r$, and all $\mu^0\in\Pz(X)$ with $W\big(\mu^0, \delta_x\big)<r^4$  there exists a $W$-geodesic $(\mu^a)_{a\in[-1,1]}$ with $W\big(\mu^1, \delta_y\big)\le r^2$
 and
\begin{equation}\label{robust}
S(\mu^1)-2S(\mu^0)+S(\mu^{-1})\le \overline K(r)\cdot r^2.
\end{equation}
The function $\overline K$ is called the approximation function.
\end{definition}

\begin{theorem}
Let $(M,g,f)$ be a weighted Riemannian manifold  with Riemannian curvature bounded 
by $\sigma$ and injectivity radius bounded from below by $\delta>0$.

Then the mm-space $(X,d,m)$  induced by $(M,g,f)$ has a robust synthetic upper Ricci bound $K$ if and only if $$\Ric_f(\xi,\xi)\le K\cdot |\xi|^2\qquad(\forall \xi\in TM).$$
In this case, the approximation function can be chosen as 
$$\overline K(r)=K+\sigma^2\, r^2$$
for  all $r\in (0,R)$ with  suitable $R=R(\sigma,\delta)>0$ 
and $\overline K(r)=\infty$ else.
\end{theorem}

\begin{proof}
{\it  ``Only if''-implication.} Assume that $\Ric_f\le K\cdot g$ is not true. That is, $\Ric_f(\xi,\xi)\ge K+2\epsilon$
for some $x\in X$, some $\xi\in T_xM$ with $|\xi|=1$, and some $\epsilon>0$. Smoothness of the data implies that there exists $r_0>0$ such that
$\Ric_f(\dot\gamma,\dot\gamma)\ge K+\epsilon$ for all unit speed geodesics $(\gamma^a)_{[-r,r]}$ with $\gamma^0\in B_{r^4}(x)$ and $\gamma^r\in B_{2r^2}(\exp_x(r\xi))$ for $r<r_0$. 
Now consider any unit speed $W$-geodesic $(\mu^a)_{[-r,r]}$ with $\mu^0$ being the uniform distribution in $B_{r^4}(x)$ and $W(\mu^r,\delta_y)\le r^2$ where $y=\exp_x(r\xi)$. Boundedness of the Riemannian curvature implies that
$\supp[\mu^r]
\subset B_{2r^2}(\exp_x(r\xi))$
provided $r$ is small enough.
Thus 
$$2S(\mu^{0})\le S(\mu^{-r})+S(\mu^r)-\frac{K+\epsilon}4W^2(\mu^{-r},\mu^r).$$ 
This contradicts \eqref{robust}.

{\it  ``If''-implication.} Let $(M,g,f)$ be given with $K$, $\sigma$ and $\delta$ as above.
Then there exist constants $R=R(\sigma,\delta)>0$ and $C=C(\sigma,\delta)$ such that $\forall r<R, \forall x\in M, v\in T_xM$ with $|v|=1$: $\exists \phi\in {\mathcal C}^2_{comp}(M)$ with the following properties
\begin{enumerate}
\item 
$\supp[\phi]\subset B_{3r}(x)$,
$\nabla\phi(x)=v$,
$|\nabla \phi|\le 1$ on $M$,
$|\Hess \phi|\le C$ on $M$
\item
$|\Hess \phi|\le C r$ in the $r^2$-neighborhood of the geodesic $\big( \exp_x(s\nabla\phi)\big)_{s\in[-r,r]}$
\item for $s\in(-r,r)$,
the map $F_s: z\mapsto \exp_z(s\nabla\phi)$ maps 
$B_{r^2/4}(x)$ into $B_{r^2/2}\big(F_s(x)\big)$.
\end{enumerate}
Indeed, we can choose $R=(\frac{\delta\wedge 1}6)^{4}$
and $\phi$ to be a  smoothened and truncated modification (supported within $B_{3r}(x)$) of  the function
$\phi_0(z)=d(z,o_+)-d(z,o_-)$ where $o_\pm=\exp_x(\pm R^{1/4} v)$.
Observe that $\nabla\phi_0(x)=v$,
$|\nabla \phi_0|\le 1$ on $M$ and by comparison geometry
$$\sqrt\sigma\cot(\sqrt\sigma t_+(z))- \sqrt\sigma\coth(\sqrt\sigma t_-(z))\le \Hess \phi_0
\le \sqrt\sigma\coth(\sqrt\sigma t_+(z))- \sqrt\sigma\cot(\sqrt\sigma t_-(z))
$$ for all $z\in M$
where $t_\pm(z)=d(z,o_\pm)$.
Thus for $z\in B_{r^2}\big(\bigcup_{s\in[-r,r]}F_s(x)\big)$ and sufficiently small $R$ (i.e. $R<R_0(\sigma)$)
\begin{eqnarray*}|\Hess \phi_0|&\le& 
\Big|\frac1{t_+(z)}-\frac1{t_-(z)}\Big|+ \sigma \big|t_+(z)-t_-(z)\big|\\
&\le& 
\frac1{R^{1/4}+s-r^2}-\frac1{R^{1/4}+s+r^2}+ 2\sigma (r+ r^2)\\
&\le&
\frac{2r^2}{(R^{1/4}-r)^2-r^4}+ 2\sqrt\sigma (r+ r^2)\le C r.
\end{eqnarray*}
Finally note  that for $s>0$ the image of the ball $B_{r^2/4}(x)$ under $F_s$ is contained in 
the image under the map 
$z\mapsto \exp_z (s\nabla d(.,o_-))$
which in turn (again by comparison geometry) is contained in the ball around $F_s(x)$ with radius
$$\rho=\sinh(\sqrt\sigma( \sinh^{-1}(r^2/(4\sqrt\sigma)))\le r^2/2$$
for $r$ sufficiently small (i.e. $r<R_1(\sigma)$).

\medskip

Given a probability measure $\mu^0$ with $W\big(\mu^0,\delta_x)<r^4$, define the $W$-geodesic
$(\mu^a)_{a\in[-r,r]}$ by
$$\mu^a=\big( \exp(a\phi^0)\big)_* \,\mu^0.$$
We decompose this geodesic of probability measures into two geodesics of subprobability measures $\mu^a=\hat\mu^a+\check\mu^a=\big( \exp(a\phi^0)\big)_* \,\hat\mu^0+\big( \exp(a\phi^0)\big)_* \,\check\mu^0$ where
$\hat\mu^0=1_{B_{r^2/4}(x)}\mu^0$ and $\check\mu^0=1_{X\setminus B_{r^2/4}(x)}\mu^0$.

Then
$c:=\check\mu^r(X)=\check\mu^0(X)\le W^2\big(\mu^0,\delta_x)/(r^2/4)^2\le 16 r^4$
and 
the subprobability measure $\hat\mu^r$ will be supported in $B_{r^2/2}(y)$. 
Moreover, $W^2\big(\check\mu^r,c\delta_y)\le 3W^2\big(\check\mu^0,c\delta_x)+6cr^2\le 3r^8+6\cdot16 r^4\cdot r^2$
and thus
$$W^2\big(\mu^r,\delta_y)\le 3r^8+96 r^6+r^2/2\le r^2.$$

Moreover, following the calculations in the proof of Theorem \ref{thm-sharp}, we obtain for a.e.\ $s\in [0,r]$
\begin{eqnarray*}
\partial_a S(\hat\mu^a)\big|_{a=s}-\partial_a S(\hat\mu^a)\big|_{a=-s}
&=&
\int\Delta_f\phi^{-s}d\hat\mu^{-s}-\int\Delta_f\phi^{s}d\hat\mu^{s}\\
&=&\int_{-s}^s\int\Big[
\Ric_f(\nabla\phi^a,\nabla\phi^a)+\big\|\D^2\phi^a\big\|_{2,2}^2\Big]d\hat\mu^a\,da\\
&\le& 2s\cdot\big(K+\sigma\tan^2(\sqrt\sigma (s+c_0))\big)\cdot \hat\mu^0(X)
\end{eqnarray*}
with $c_0:=r^2/\sigma$ such that
$\big\|\D^2\phi^0\big\|_{2,2}\le\sqrt\sigma\tan(\sqrt\sigma \, c_0)$ on 
 $B_{r^2}\big(\bigcup_{s\in[-r,r]}F_s(x)\big)$.
 
 Similarly,
 \begin{eqnarray*}
\partial_a S(\check\mu^a)\big|_{a=s}-\partial_a S(\check\mu^a)\big|_{a=-s}
&\le& 2s\cdot\big(K+\sigma\tan^2(\sqrt\sigma (s+c_1))\big)\cdot \check\mu^0(X)
\end{eqnarray*}
with $c_1:=C/\sigma$ such that
$\big\|\D^2\phi^0\big\|_{2,2}\le\sqrt\sigma\tan(\sqrt\sigma \, c_1)$ on $M$.

Thus
\begin{eqnarray*}
S(\hat\mu^r)+S(\hat\mu^{-r})-2S(\hat\mu^0)&=&
\int_0^r\Big[
\partial_a S(\hat\mu^a)\big|_{a=s}-\partial_a S(\hat\mu^a)\big|_{a=-s}
\Big]ds\\
&\le&r^2\cdot\big(K+\sigma\tan^2(\sqrt\sigma (r+r^2/\sigma))+16r^4
\sigma\tan^2(\sqrt\sigma (r+C/\sigma))
\big)\\
&\le& r^2\cdot (K+2\sigma^2r^2)
\end{eqnarray*}
for $r<R_2(\sigma, C)$.
This proves the claim.
\end{proof}

\begin{theorem} Let $\overline K$  be given with $\lim_{r\to0}\overline K(r)=K$.
Assume that  $(X_n,d_n,m_n), n\in\N$, is a sequence of mm-spaces with normalized volume and the properties
\begin{itemize}
\item for each $n$, the mm-space $(X_n,d_n,m_n)$ has robust synthetic upper Ricci bound $K$ with approximation function $\overline K$;
\item for $n\to\infty$, the mm-spaces $(X_n,d_n,m_n)$ converge to a locally compact mm-space $(X,d,m)$ w.r.t. the metric $\mathbb D=\mathbb D_2$ introduced in \cite{SturmActa1}.
\end{itemize}
Then the mm-space $(X,d,m)$ 
 has robust synthetic upper Ricci bound $K$ with approximation function $\overline K$.
\end{theorem}

Recall that $\mathbb D$-convergence means that
 for each $n$ there exists an isometric embedding of $(X_n,d_n)$
and $(X,d)$ into a mm-space $(\hat X_n,\hat d_n)$ such that
$$\hat W_{n}(\hat m_n,\hat m)\to0$$
as $n\to\infty$ where $\hat W_n$ denotes the $L^2$-Kantorovich-Wasserstein distance derived from the metric $\hat d_n$ and $\hat m_n,\hat m$ denote the push forwards of the measures $m_n,m$ by the embedding maps.

In particular, recall that mGH-convergence implies $\mathbb D$-convergence.

\begin{proof}
Given $r,x,y$ and $\mu^0$ as requested on $X$, choose $\epsilon>0$ such that 
$W(\mu^0,\delta_x)\le (r-\epsilon)^4-4\epsilon$.
For each $n$, define the measure $\mu^0_n$ on $X_n$ by means of the transportation map from \cite{SturmActa1}. Then for $n$ large enough
$\hat W_n(\mu^0,\mu^0_n)\le\epsilon$.
To simplify notation, here and in the sequel we will identify the measures $\mu^0$ and $\mu^0_n$ with their push forwards $\hat\mu^0$ and $\hat\mu^0_n$ under the embedding maps.

Put $\delta_n=\mathbb D\big(X,X_n\big)$ and
$v=m(B_\epsilon(y)$.
Then
$$\inf_{(z,z_n)\in B_\epsilon(y)\times X_n} \hat d^2(z,z_n)\cdot v\le \delta_n^2.$$
Thus there exists $y_n\in X_n$ with $\hat d_n(y,y_n)\le \epsilon+\frac1{\sqrt v}\delta_n$.
And for sufficiently large $n$ therefore $\hat d_n(y,y_n)\le 2\epsilon$.
Similarly, there exists $x_n\in X_n$ with $\hat d_n(x,x_n)\le 2\epsilon$ for sufficiently large $n$.

Put $r_n=d_n(x_n,y_n)$. Then $r-4\epsilon\le r_n\le r+4\epsilon$ and thus
$$W_n\big(\mu^0_n,\delta_{x_n}\big)\le
\hat W_n\big(\mu^0_n,\mu^0\big)+W\big(\mu^0,\delta_{x}\big)+\hat d_n(x,x_n)\le
\epsilon+[(r-4\epsilon)^4-4\epsilon]+2\epsilon<r_n^4.$$
The assumption of robust curvature for $(X_n,d_n,m_n)$ thus implies that  there exists a $W_n$-geodesic $(\mu_n^a)_{a\in[-1,1]}$ in $\Pz(X_n)$,
depending on the choice of $\epsilon$ (but for the moment we suppress this dependence in the notation),
such that
\begin{equation}
\frac1{r_n^2}\Big(S_n(\mu_n^1)-2S_n(\mu_n^0)+S_n(\mu_n^{-1})\Big)\le \overline K(r_n).
\end{equation}
and
$W_n\big(\mu^1_n,\delta_{y_n}\big)\le r_n^2\le (r+4\epsilon)^2$
which implies
$$\hat W_n\big(\mu^1_n,\delta_{y}\big)\le  (r+4\epsilon)^2+2\epsilon.$$

By local comapctness of $X$ and $W$-boundedness of the  involved measures, for $n\to\infty$ the  geodesics $(\mu_n^a)_{a\in[-1,1]}$ in $\Pz(X_n)$ will converge (via transport maps and/or embedding maps) to a geodesic $(\mu^a)_{a\in[-1,1]}$ in $\Pz(X)$ with the orginally given $\mu^0$ as its midpoint and with 
$W\big(\mu^1, \delta_y\big)\le (r+4\epsilon)^2+2\epsilon$.
More precisely,  for each $\epsilon>0$ there exists such a geodesic
$(\mu^{\epsilon,a})_{a\in[-1,1]}$ in $\Pz(X)$. Passing again to a suitable subsequential limit we obtain a geodesic  $(\mu^a)_{a\in[-1,1]}$ in $\Pz(X)$ with the originally given $\mu^0$ as its midpoint and with 
$W\big(\mu^1, \delta_y\big)\le r^2$.

By construction, 
$S_n(\mu_n^0)\to S(\mu^0)$
and thus by lower semicontinuity of $n\mapsto S_n(\mu_n^{\epsilon,a})$ and 
 $\epsilon\mapsto S(\mu^{\epsilon,a})$ for  $a\in \{-1,1\}$
\begin{eqnarray*}
\frac1{r^2}\Big(S(\mu^1)-2S(\mu^0)+S(\mu^{-1})\Big)&\le&
\liminf_{\epsilon\to0}
\frac1{r^2}\Big(S(\mu^{\epsilon,1})-2S_n(\mu_n^0)+S(\mu^{\epsilon,-1})\Big)\\
&\le&
\liminf_{\epsilon\to0} \ \liminf_{n\to\infty}
\frac1{r_n^2}\Big(S_n(\mu_n^{\epsilon,1})-2S_n(\mu_n^0)+S_n(\mu_n^{\epsilon,-1})\Big)\\
&\le&\liminf_{n\to\infty} \overline K(r_n)\le \overline K(r).
\end{eqnarray*}
This proves the claim.
\end{proof}

\begin{remark} 
The Eguchi-Hanson metric \cite{Eguchi-Hanson} is a Ricci flat metric on $T^*S^2$, the cotangent bundle of the 2-sphere,  with the property that  blowdown sequences converge  
to $\R^4/\Z_2$, the orbifold obtained by identifying  antipodal points, or in other words, to the metric cone on $\R{\mathbb P}(3)$.
The convergence is in the pointed Gromov-Hausdorff sense and  away from the isolated point singularity it is also in the smooth Cheeger-Gromov sense.
The limit space is a  mm-space with a conical singularity and outside of which it is flat. 
According to \cite{ErSt}, no (weak or strong) synthetic upper Ricci bounds for this mm-space will exist. 

Moreover, the limit space also admits \emph{no robust upper bound} for the Ricci curvature.
Actually, the unbounded positive curvature in the singular points of the limit space, will not be detected by non-semiconcavity of the entropy  along optimal transports. However, the concept of robust Ricci bounds in addition requires a certain kind of extendability of geodesics which will be violated for the limit space. 

Thus this might look like a counterexample to our stability result for synthetic upper Ricci bounds.
Our stability result, however, requires that the approximating spaces satisfy  a  kind of extendability of geodesics (more precisely, extendability of Wasserstein geodesics) to a \emph{uniform extent} which will be violated in this case for geodesics through the origin.
\end{remark}

\end{document}